\documentclass[10pt]{article}
\usepackage{color}
\usepackage{amssymb}
\usepackage{amsthm,array,amssymb,amscd,amsfonts,latexsym, url}
\usepackage{amsmath}
\usepackage[all]{xy}
\newtheorem{theo}{Theorem}[section]
\newtheorem{prop}[theo]{Proposition}

\newtheorem{lemm}[theo]{Lemma}
\newtheorem{coro}[theo]{Corollary}
\newtheorem{rema}[theo]{Remark}
\newtheorem{Defi}[theo]{Definition}

\newtheorem{example}[theo]{Example}
\newtheorem{question}[theo]{Question}

\title{Varieties with representable ${\rm CH}_0$-group and a question of Colliot-Th\'{e}l\`{e}ne}
\author{Claire Voisin\footnote{The author is supported by the ERC Synergy Grant HyperK (Grant agreement No. 854361).}}
\date{}

\newfont{\gothic}{eufb10}
\begin{document}

\maketitle
\begin{abstract}  We continue our investigation of the geometry of the Albanese  morphism on  0-cycles. We provide an example of a smooth projective variety with representable ${\rm CH}_0$-group  but with no universal $0$-cycle, which answers a question asked by Colliot-Thélène. Our construction relies on a counterexample to the integral Hodge conjecture provided by Benoist and Ottem.
 \end{abstract}
\section{Introduction}
Let $X$ be  a smooth projective complex variety of dimension $n$. The Albanese morphism
\begin{eqnarray}\label{eqalbmoroh} {\rm alb}_X: X\rightarrow {\rm Alb}(X)\end{eqnarray} is a morphism of algebraic varieties defined using a base-point in $X$. The induced morphism
\begin{eqnarray}\label{eqaX} a_X: {\rm CH}_0(X)_{\rm hom}\rightarrow  {\rm CH}_0({\rm Alb}(X))_{\rm hom}\rightarrow {\rm Alb}(X), \end{eqnarray}
where the last map is the sum map, is  the Griffiths-Abel-Jacobi map for zero-cycles on $X$, and can also be characterized as the universal functorial regular abelian quotient of ${\rm CH}_0(X)_{\rm hom}$ (see \cite{murre}, \cite{serre}). Indeed, the group morphism $a_X$ is surjective and regular, which means that for any smooth projective variety $W$ with base-point $w_0$ and codimension $n$ cycle $\Gamma\in{\rm CH}^n(W\times X)$, the map
$$W\rightarrow {\rm Alb}(X),$$
$$w\mapsto a_X\circ\Gamma_*(w-w_0)$$
is a morphism of algebraic varieties. In \cite{voisinArbarello}, \cite{voisincollino}, we introduced and studied the notion of ``universal $0$-cycle'' (parameterized by the Albanese variety), with the following
\begin{Defi} A universal $0$-cycle for $X$ is a codimension $n$ cycle $\Gamma\in{\rm CH}^n({\rm Alb}(X)\times X)$, where $n:={\rm dim}\,(X)$, such that for any $t\in {\rm Alb}(X)$
$$a_X\circ \Gamma_*(\{t\}-\{0_{{\rm Alb}(X)}\})=t\,\, {\rm in}\,\,{\rm Alb}(X).$$
\end{Defi}
\begin{example}\label{example}{\rm A smooth projective curve admits a universal $0$-cycle given by the Poincar\'{e} divisor on $J(C)\times C $.}
\end{example}
We will review in Section \ref{sec1} some basic facts concerning this definition. Let us just say that there always exists a universal $0$-cycle with $\mathbb{Q}$-coefficients, and that an $X$ not having a universal $0$-cycle provides a counterexample to the integral Hodge conjecture  on ${\rm Alb}(X)\times X$.
Furthermore, we proved in \cite{voisincollino} that  there exist smooth projective varieties not admitting a universal $0$-cycle.

Recall that, following Mumford \cite{mumford} and Roitman \cite{roitman},  $X$ has a representable ${\rm CH}_0$-group if the  morphism $a_X$ of (\ref{eqaX})  is an isomorphism.
The examples provided in \cite{voisincollino} are surfaces whose Albanese variety is isomorphic to the intermediate Jacobian of a rationally connected threefold. These surfaces do not have in general  a representable ${\rm CH}_0$-group (like the surface of lines of  a smooth cubic threefold, they rather tend to be of  general type).
In the paper \cite{colliotTh}, Colliot-Th\'{e}l\`{e}ne asked the following
\begin{question}\label{questionCT}  Does there exist a smooth complex  projective variety $X$ with representable ${\rm CH}_0$ group and not having a universal $0$-cycle?
\end{question}
The purpose of this paper is to present such an example.
\begin{theo}\label{theo} There exists a smooth projective  threefold $X$ such that ${\rm CH}_0(X)$ is representable and $X$ does not admit a universal $0$-cycle.
\end{theo}

Our examples have some torsion in their integral Betti cohomology, so one might ask whether such examples exist  without torsion in their integral Betti cohomology. From the viewpoint of the integral Hodge conjecture, they also provide an example of a curve class (that is, a Hodge homology class of degree $2$) on a fourfold with torsion canonical bundle, that is not algebraic modulo torsion. Benoist and Ottem construct in \cite{benoistOttem} examples of non-algebraic torsion curve classes on a threefold with torsion canonical bundle.
Our construction  in fact relies on the  results of  \cite{benoistOttem}. In Section \ref{secBO}, we will provide a mild generalization and a short proof of one of their statements (see Proposition \ref{propBOwithnewproof}).

The assumptions of Theorem \ref{theo} are very restrictive geometrically. Indeed, we have
   \begin{lemm} \label{lestructure} Let $X$ be a smooth projective complex variety  with representable ${\rm CH}_0$ group and nontrivial Albanese variety. Then there exist a smooth projective curve $C$ such that $J(C)\cong {\rm Alb}(X)$ and a dominant morphism $\phi_C: X\rightarrow C$ such that the Albanese morphism of $X$ is the composition of $\phi_C$ and the Albanese morphism of $C$.
   \end{lemm}
\begin{proof} When ${\rm CH}_0(X)$ is representable, Mumford's theorem \cite{mumford} (see also \cite{blochsri} for an alternative proof), tells that $H^0(X,\Omega_X^i)=0$ for $i\geq 2$. It follows that the Albanese morphism
${\rm alb}_X: X\rightarrow {\rm Alb}(X)$ of (\ref{eqalbmoroh}) has rank $\leq 1$ everywhere, hence factors (using the Stein factorization) as the composition of a morphism $\phi_C: X\rightarrow C$ and a morphism $\psi: C\rightarrow {\rm Alb}(X)$. Using the universal property of the Albanese morphism for $X$ and $C$, we conclude that
$\psi$ induces an isomorphism ${\rm Alb}(C)\cong {\rm Alb}(X)$.
\end{proof}
\begin{rema}{\rm As the proof above  shows, the conclusion  of Lemma \ref{lestructure} holds under the asssumption that $h^{i,0}(X)=0$ for $i\geq 2$.}
\end{rema}
The fact that ${\rm Alb}(X)$ is isomorphic to the Jacobian of a curve makes a priori easier for $X$ to have a universal $0$-cycle. Indeed, for a Jacobian $J(C)$, or for any abelian variety which is a direct summand in a Jacobian, we can use  the universal Poincar\'{e} divisor on $J(C)\times C$ and this reduces the problem of constructing a codimension $n$ cycle in ${\rm Alb}(X)\times X$ to the problem of constructing a codimension $n$ cycle in $C\times X$.

We will discuss in Section \ref{sec1} various properties of a smooth projective variety $X$ related to the existence of a universal $0$-cycle (see Proposition \ref{prokununiv}).
 We will relate in particular the existence of a universal $0$-cycle for $X$ with the structure of the K\"{u}nneth projector $\delta_1$ on $X$. Proposition \ref{prokununiv} says that the property of having a universal $0$-cycle has a good formulation in terms of  the motive of $X$ when one knows furthermore  that the Albanese variety of $X$ is a  direct summand in the Jacobian of a curve, which is the case when $X$ has representable ${\rm CH}_0$ by Lemma \ref{lestructure}. Note that, thanks to \cite{degaay} and the recent work \cite{SEggay}, it is now known that an abelian variety is not always a direct summand in the Jacobian of a curve (equivalently by \cite{degaay} and \cite{voisinArbarello},  the integral Hodge conjecture can be wrong for curve classes on abelian varieties).
Motivated by Lemma \ref{lestructure}, we will  consider more generally  in Section \ref{sec1feberedvar} smooth projective varieties $X$ whose Albanese map factors as the composition of a dominant  morphism
$$\phi_X:X\rightarrow C,$$
and the Abel map of $C$, as it is the case when ${\rm CH}_0(X)$ is representable  by Lemma \ref{lestructure}.
Recall that the index of a dominant morphism $X\rightarrow B$, where $B$ is irreducible, is the greatest common divisor of the degrees ${\rm deg}(Z/B)$ for all $Z\subset X$ closed irreducible dominant over $B$ and of dimension equal to ${\rm dim}\,B$. In the situation above, we will discuss the relation between the index of $\phi_X$ being $1$ and the existence of a universal $0$-cycle for $X$ (see Proposition \ref{coroaprouver}).
In one direction, we will prove  that, if the Albanese map of $X$  factors through a morphism
$ \phi_X:X\rightarrow C $ as in  Lemma \ref{lestructure} and $\phi_X$  has index $1$, $X$ has a universal $0$-cycle.
This statement has no converse, as
we will also exhibit an example where the index of $\phi_X$ is not equal to $1$, but $X$ has a universal $0$-cycle.

We will prove Theorem \ref{theo} in Section \ref{secexCT}. In Section \ref{secsurfcaes}, we will discuss the case of surfaces with representable ${\rm CH}_0$ group, where more can be said thanks to the structure of surfaces with $q\not=0$, $p_g=0$. Such surfaces $X$ are quotients of products of curves by a finite group, and we prove that the index of  $\phi_X$  is $1$ in the case of a cyclic group.

\vspace{0.5cm}

{\bf Thanks.} {\it I thank Jean-Louis Colliot-Thélène for his remarks and careful reading. I also thank the referee for his/her  useful comments and corrections.}

\section{A discussion of universal $0$-cycles \label{sec1}}

Let $X$ be a smooth projective variety of dimension $n$ over the complex numbers. The Albanese variety of $X$ is an abelian variety, which as a complex torus is constructed as the quotient $${\rm Alb}(X)=H^{n}(X,\Omega_X^{n-1})/H^{2n-1}(X,\mathbb{Z})_{\rm tf}=H^{0}(X,\Omega_X)^*/H_1(X,\mathbb{Z})_{\rm tf},$$
where in both expressions, ``tf'' means ``modulo torsion''.
It follows that the Albanese map induces by construction an isomorphism of lattices
\begin{eqnarray}\label{eqcohiso} H_1(X,\mathbb{Z})_{\rm tf}\cong H_1({\rm Alb}(X),\mathbb{Z}).
\end{eqnarray}
Here all the cohomology groups are integral Betti cohomology groups of the corresponding complex manifolds.
The isomorphism (\ref{eqcohiso}) has an inverse in ${\rm Hom}(H_1({\rm Alb}(X),\mathbb{Z}),H_1(X,\mathbb{Z})_{\rm tf}) $ which provides an element $\gamma_X\in H^1({\rm Alb}(X),\mathbb{Z})\otimes H^{2n-1}(X,\mathbb{Z})_{\rm tf}$.
As the isomorphism (\ref{eqcohiso}) is an isomorphism of Hodge structures, the class $\gamma_X$ provides by K\"{u}nneth decomposition  a Hodge class of degree $2n$ on ${\rm Alb}(X)\times X$, also denoted $\gamma_X$. The existence of a universal $0$-cycle on $X$ is equivalent to the fact that $\gamma_X$ is algebraic, or rather that there exists an algebraic cycle $\Gamma$ of codimension $n$ in ${\rm Alb}(X)\times X$, such that the K\"{u}nneth component of type $(1,2n-1)$ of $[\Gamma]$ equals $\gamma_X$. In particular, if a smooth projective variety $X$ has no universal $0$-cycle, then the class $\gamma_X$ on  ${\rm Alb}(X)\times X$  provides a counterexample to the integral Hodge conjecture. More precisely, it provides a counterexample to the integral  Hodge conjecture modulo torsion, since the class $\gamma_X$ is well-defined in $H^{2n}({\rm Alb}(X)\times X,\mathbb{Z})_{\rm tf}$ and is not algebraic there when $X$ has no universal $0$-cycle.

If $X,\,Y$ are smooth projective complex varieties, with ${\rm dim}\,X=n$, and $Z\in{\rm CH}^n(Y\times X)$, the morphism of Hodge structures
$$[Z]_*: H_1(Y,\mathbb{Z})_{\rm tf}\rightarrow H_1(X,\mathbb{Z})_{\rm tf}$$
induces a morphism of complex tori (which are in fact abelian varieties) that we will also denote
$$[Z]_*:{\rm Alb}(Y)\rightarrow {\rm Alb}(X).$$

We will use the following construction.
\begin{lemm} \label{lerefsuggestion} Let $X$ be  a smooth projective variety and $C$ a smooth curve with a morphism
$j: C\rightarrow X$ inducing
$$j_*:J(C)\rightarrow {\rm Alb}(X).$$
Assume there exists a morphism of abelian varieties $s:{\rm Alb}(X)\rightarrow J(C)$ such that
\begin{eqnarray} \label{eqjstar1}s\circ j_*= k{\rm Id}_{{\rm Alb}(X)}\end{eqnarray} for some integer $k$. Then there exists a codimension $n$ cycle
$$\Gamma\in{\rm CH}^n({\rm Alb}(X)\times X)$$
such that
$$[\Gamma]_*=k{\rm Id}_{{\rm Alb}(X)}:{\rm Alb}({\rm Alb}(X))={\rm Alb}(X)\rightarrow {\rm Alb}(X).$$
\end{lemm}
\begin{proof}
The curve $C$ admits a universal $0$-cycle $\Gamma_C\in {\rm CH}^1(J(C)\times C)$ (see Example \ref{example}). Then
the codimension $n$-cycle
$$\Gamma_X:=(\overline{s},{\rm Id}_X)^*(({\rm Id}_{J(C)},j)_*(\Gamma_C))\in{\rm CH}^n({\rm Alb}(X)\times X)$$
 has by (\ref{eqjstar1}) the property that
$$ [\Gamma_X]_*=k {\rm Id}_{{\rm Alb}(X)}:{\rm Alb}({\rm Alb}(X))={\rm Alb}(X)\rightarrow {\rm Alb}(X).$$
\end{proof}
As a first application, let us prove that any smooth projective variety $X$ has a universal $0$-cycle with rational coefficients. Indeed, choosing a smooth  curve $C$ which is a complete intersection of ample hypersurfaces in $X$, and denoting $j: C\rightarrow X$ the inclusion map,
$$j_*:H_1(C,\mathbb{Z})\rightarrow H_1(X,\mathbb{Z})_{\rm tf}$$
is a surjective morphism of Hodge structures. By semisimplicity of polarized weight $1$ Hodge structures,
there exists a morphism of Hodge structures
$$ s: H_1({\rm Alb}(X),\mathbb{Z})\cong  H_1(X,\mathbb{Z})_{\rm tf}\rightarrow H_1(C,\mathbb{Z})$$
such that \begin{eqnarray} \label{eqjstar} j_*\circ s=k{\rm Id}_{H_1(X,\mathbb{Z})_{\rm tf}}\end{eqnarray}
for some nonzero integer $k$. The morphism $s$ induces a morphism
$$\overline{s}: {\rm Alb}(X)\rightarrow J(C)$$
of abelian varieties and $j_*,\,\overline{s}$ satisfy the assumptions of Lemma \ref{lerefsuggestion}. Hence Lemma
 \ref{lerefsuggestion} provides a universal $0$-cycle with rational coefficients.

let us now discuss the link between  universal $0$-cycle and Künneth decomposition of the Albanese motive. Assume that the    smooth projective complex variety $X$ of dimension $n$ has the property  that $H^*(X,\mathbb{Z})$ has no torsion. The K\"{u}nneth components $\delta_i\in H^i(X,\mathbb{Z})\otimes H^{2n-i}(X,\mathbb{Z})$ of the diagonal of $X$ are then well defined, and it is known that $\delta_1$ is algebraic (see \cite{murredec}), at least with $\mathbb{Q}$-coefficients.
\begin{prop} \label{prokununiv}The following statements are equivalent:

(i)  The K\"{u}nneth projector $\delta_1$ is algebraic on $X\times X$ and is the class of a cycle supported on $X\times C$ for some curve $C\subset X$.

(ii) $X$  has a universal $0$-cycle and the Albanese variety of $X$ is a direct summand in the Jacobian of a curve.

(iii) The motive of $X$ contains a direct summand $M_1$, with the property that $M_1$ is a direct summand in the motive of a curve and ${\rm Alb}(X)={\rm Alb}(M_1)$.

\end{prop}
In the statement above, we allow nonconnected curves, so ``Jacobian of a curve" stands for ``product of Jacobians of curves".
\begin{rema}{\rm If $X$ has representable ${\rm CH}_0$ group, then it is automatic that the Albanese variety of $X$ is a direct summand in the Jacobian of a curve by Lemma \ref{lestructure}, so in this case we see that $X$ has a universal $0$-cycle if and only if the motive of $X$ is the direct sum of a motive with  ${\rm CH}_0=0$ and a motive which is a direct summand in the motive of a curve.}
\end{rema}

\begin{proof}[Proof of Proposition \ref{prokununiv}] Assume there exist a  curve $C\subset X$ and a codimension $1$ cycle $Z\subset X\times C$ such that the class of $Z$ in $X$ is $\delta_1$. We can replace by desingularization  $C$ by a smooth curve $\widetilde{C}$ with a morphism $\tilde{j}:\widetilde{C}\rightarrow X$, such that $Z$ lifts to a codimension $1$ cycle $\widetilde{Z}\subset X\times C$. Then, as the class of $Z$ in $X$ is $\delta_{1}$, we have for any $\alpha\in H_1(X,\mathbb{Z})$
\begin{eqnarray}\label{eqpourpreubeproofdirect} \alpha=\delta_{1*}(\alpha)=\tilde{j}_*([\widetilde{Z}]_*(\alpha))\,\,{\rm in}\,\,H_1(X,\mathbb{Z}).
\end{eqnarray}
It follows from (\ref{eqpourpreubeproofdirect}) that, via $\tilde{j}_*:J(\widetilde{C})\rightarrow {\rm Alb}(X)$, ${\rm Alb}(X)$ is a direct summand of $J(\widetilde{C})$, with right  inverse given by $ \sigma:=[\widetilde{Z}]_*:{\rm Alb}(X)\rightarrow J(\widetilde{C})$, which proves the second statement in (ii).
Next, we apply Lemma \ref{lerefsuggestion} to $\tilde{j}$ and $\sigma$ and consider
 the cycle
$$ \Gamma_1:= (\sigma,{\rm Id}_C)^*\circ({\rm Id}_{J(C)},\tilde{j})_*(P_{\widetilde{C}})\in {\rm CH}^n({\rm Alb}(X)\times X ).$$
It has the property that $[\Gamma_{1}]_*:{\rm Alb}({\rm Alb}(X))={\rm Alb}(X)\rightarrow {\rm Alb}(X)$ is the identity, since $\tilde{j}_*\circ \sigma={\rm Id}_{{\rm Alb}(X)}$. Hence $X$ has a universal $0$-cycle. Thus (i) implies (ii).

 In the other direction, let $\Gamma\in {\rm CH}^n( {\rm Alb}(X)\times X)$ be a universal $0$-cycle for $X$. This means that
$$[\Gamma]_*:H_1({\rm Alb}(X),\mathbb{Z})\rightarrow  H_1(X,\mathbb{Z})$$
is the natural isomorphism. Let $D$ be a (nonnecessarily connected) smooth projective  curve such that ${\rm Alb}(X)$   is a direct summand of $J(D)$, and let $$\sigma: {\rm Alb}(X)\rightarrow J(D),\,\pi:J(D)\rightarrow {\rm Alb}(X)$$ be such that $\pi\circ \sigma={\rm Id}_{{\rm Alb}(X)}$. Finally, let $i_D: D\rightarrow J(D)$ be the natural embedding and
consider the cycle
$$\Gamma_1:=    \Gamma\circ \pi\circ i_D\in{\rm CH}^n(D\times X).$$
We have $$[\Gamma_1]_*=[\Gamma]_*\circ \pi_*\circ i_{D*}=\pi: J(D)\rightarrow {\rm Alb}(X).$$
 Next, let $P_D\in {\rm CH}^1( J(D)\times D)$ be a universal $0$-cycle (or Poincar\'{e} divisor) for $D$ and consider next the cycle
$$\Gamma_2=P_D\circ \sigma \circ {\rm alb}_X\in {\rm CH}_n(X\times D).$$
We have
$[\Gamma_2]_*=\sigma:{\rm Alb}(X)\rightarrow J(D)$. It follows that
$\Gamma_X:=\Gamma_1\circ \Gamma_2\in {\rm CH}^n(X\times X)$
has the property that
\begin{eqnarray}\label{eqinversefacteur} [\Gamma_X]_*=[\Gamma_1]_*\circ [\Gamma_2]_*=\pi\circ \sigma={\rm Id}_{\rm Alb}(X).\end{eqnarray}
Finally $\Gamma_X$ is supported on $X\times D'$ where $D'\subset X$ is the curve in $X$ defined as ${\rm pr}_2({\rm Supp}(\Gamma_1)$. Furthermore one easily checks that $[\Gamma_X]^*=0$ on $H^i(X,\mathbb{Z})$ for $i\not =1$. It follows that $[\Gamma_X]=\delta_1$.
This proves that (ii) implies (i).

The equivalence of  (ii) with property (iii) is  now clear. If we have a direct summand $M_1$ as in (iii), then $M_1$ has a universal $0$-cycle, being a direct summand in the motive of a curve $C$, and its Albanese variety is a direct summand of $J(C)$. Hence (ii) holds. Conversely, if (ii) holds, then we construct $\Gamma_1$ and $\Gamma_2$ as above and (\ref{eqinversefacteur}) exactly says that $[\Gamma_2]$ realizes motivically  $H_1(X,\mathbb{Z})$ as a direct summand in $H_1(D,\mathbb{Z})$.
\end{proof}

\subsection{Previous examples of smooth projective varieties with no universal $0$-cycle}
The first example of a smooth projective surface not admitting a universal $0$-cycle is constructed in \cite{voisincollino}. It builds on the fact that there exist rationally connected $3$-folds $Y$ with no universal codimension $2$ cycle parameterized by the intermediate Jacobian $J(Y)$ (such examples are constructed in \cite{voisinINVENT}). Here a universal codimension $2$ cycle $\mathcal{Z}_{\rm univ}\in{\rm CH}^2(J(Y)\times Y)$ is characterized by the condition that the composite map
$$J(Y)\rightarrow {\rm CH}^2(Y)_{\rm alg}\rightarrow  J(Y),$$
$$t\mapsto \psi_Y(\mathcal{Z}_{{\rm univ},t}- \mathcal{Z}_{{\rm univ},0_{J(Y)}})$$
is the identity, where $\psi_Y$ is the Abel-Jacobi morphism (an isomorphism in this case) of $Y$.
The second ingredient is the following result proved in \cite{voisincollino}:
\begin{theo} \label{theocollino} Given a smooth projective  rationally connected threefold $Y$ over $\mathbb{C}$, there exist a smooth projective surface $S$ and a codimension $2$ cycle $\mathcal{Z}\in  {\rm CH}^2(S\times Y)$ such that the induced Abel-Jacobi map
$$[\mathcal{Z}]_*: {\rm Alb}(S)\rightarrow J(Y)$$
is an isomorphism.
\end{theo}
The varieties $Y$ and $S$ being as above, if $Y$ does not have a universal codimension $2$ cycle, $S$ does not have a universal $0$-cycle, since a universal $0$-cycle $$\mathcal{Z}_S\in {\rm CH}^2({\rm Alb}(S)\times S)={\rm CH}^2(J(Y)\times S)$$ would produce a universal codimension $2$ cycle $$\mathcal{Z}_{\rm univ}=\mathcal{Z}\circ \mathcal{Z}_S\in{\rm CH}^2(J(Y)\times Y)$$
for $Y$.

  A particular case of this  construction also solves in the negative the following  question asked by Colliot-Th\'{e}l\`{e}ne in the original version of \cite{colliotTh}.
  \begin{question}\label{q2CT} Given  any  smooth projective variety $X$ over the complex numbers, and any function field $K=\mathbb{C}(C)$, where $C$ is a curve, is the natural morphism ${\rm CH}_0(X_K)_0\rightarrow {\rm Alb}(X)(K)$  surjective?
  \end{question}
  Indeed, we first observe the following
  \begin{prop} \label{propCA} Assume the smooth projective variety $X$ has the property that ${\rm Alb}(X)$ is a direct summand, as an abelian variety, in the Jacobian $J(C)$ of a smooth projective curve $C$. Then the following properties are equivalent.

  (i) $X$ has a universal $0$-cycle.

  (ii) For any curve $D$, the natural morphism $a_X: {\rm CH}_0(X_K)_0\rightarrow {\rm Alb}(X)(K)$ is surjective, where $K=\mathbb{C}(D)$.
  \end{prop}
  We give the proof for completeness, although similar statements and arguments can be found in \cite{voisinArbarello} and \cite{colliotTh}.
\begin{proof}   (i) implies  (ii) in an obvious way, since having a morphism $\phi: D\rightarrow {\rm Alb}(X)$ and a universal $0$-cycle $\Gamma\in {\rm CH}^n({\rm Alb}(X)\times X)$ of relative degree $0$ over ${\rm Alb}(X)$, where $n:={\rm dim}\,X$, we get by restriction to $D\times X$ a cycle
$\Gamma_D\in {\rm CH}^n(D\times X)$, of relative degree $0$ over $D$, which by definition of  a universal $0$-cycle, has the property that the composed morphism ${\rm alb}_X\circ \Gamma_{D*}:J(D)\rightarrow {\rm Alb}(X)$, equals $\phi_*:J(D)\rightarrow {\rm Alb}(X)$. It follows that, denoting by $j_D:D\rightarrow J(D)$ a chosen  inclusion, the two morphisms
${\rm alb}_X\circ \Gamma_{D*} \circ j_D: D\rightarrow {\rm Alb}(X)$ and $\phi$ differ by  a translation. Correcting $\Gamma_D$ by the adequate cycle $D\times z_0$, where $z_0\in{\rm CH}_0(X)_{\rm hom}$, gives a cycle $\Gamma'_D\in {\rm CH}^n(D\times X)$ such that
\begin{eqnarray}\label{eqnewdu1709} {\rm alb}_X\circ \Gamma_{D*} \circ j_D = \phi: D\rightarrow {\rm Alb}(X).
\end{eqnarray}
Equation (\ref{eqnewdu1709}) exactly tells that the image of $\Gamma'_D$ in ${\rm CH}_0(X_{K})_0$ maps to $\phi_K\in{\rm Alb}(X)(K)$.

In the other direction, assume that ${\rm Alb}(X)$ is a direct summand in $J(C)$, that is, for some abelian variety $B$,
$$ {\rm Alb}(X)\times B\cong J(C),$$ with inclusion $i:{\rm Alb}(X)\rightarrow J(C)$ and projection $p:J(C)\rightarrow {\rm Alb}(X)$, such that $p\circ i={\rm Id}_{{\rm Alb}(X)}$. We observe that the curve $C$ has a universal $0$-cycle $\Gamma_C$ (see Example \ref{example}).
Let $i_C$ be the natural inclusion of $C$ in $J(C)$ (it is in fact defined up to translation). The composite morphism $p_C=p\circ i_C: C\rightarrow {\rm Alb}(X)$ gives a $K$-point in ${\rm Alb}(X)(K)$, where $K=\mathbb{C}(C)$, and assuming (ii), there is a codimension $n$ cycle $Z_C\in{\rm CH}^n(C\times X)$, of relative degree $0$ over $C$, such that
\begin{eqnarray}\label{eqplouf}{\rm alb}_X\circ Z_{C*}=p_C: C\rightarrow A.\end{eqnarray}
 We next consider the composition $Z_{C,{\rm Alb}(X)}:=Z_C\circ \Gamma_C\circ i\in {\rm CH}^n({\rm Alb}(X)\times X)$ and it follows from  (\ref{eqplouf}) that ${\rm alb}_X\circ Z_{C,{\rm Alb}(X)*}={\rm Id}_{{\rm Alb}(X)}:{\rm Alb}(X)\rightarrow {\rm Alb}(X)$.
\end{proof}

We now  get a negative answer to Question \ref{q2CT} by considering a smooth projective rationally connected $3$-fold $Y$ which does not admit a universal codimension $2$ cycle while the intermediate Jacobian  $J(Y)$ has dimension $3$, so that it is the Jacobian of a curve. Such examples are constructed in \cite{voisinINVENT}: one can take for $Y$ the desingularization of a very general quartic double solid with seven nodes. We then introduce as before a surface $S$ as in Theorem \ref{theocollino}, which has ${\rm Alb}(S)\cong J(Y)$, hence has its Albanese variety isomorphic to the Jacobian of a curve. Furthermore, $S$ has no universal $0$-cycle because $Y$ does not have a universal codimension $2$ cycle, hence provides the desired example using Proposition \ref{propCA}.

\subsection{Varieties fibered over a curve\label{sec1feberedvar}}
We now consider smooth projective varieties $X$ such that the Albanese map of $X$ factors through a curve   as in  Lemma \ref{lestructure}.
 We first prove
\begin{prop}\label{coroaprouver} If the Albanese map of $X$  factors through a morphism
$ \phi_X:X\rightarrow C $, where $C$ is  a smooth projective curve such that $\phi_{X*}: {\rm Alb}(X)\rightarrow J(C) $ is an isomorphism,  and $\phi_X$  has index $1$, $X$ has a universal $0$-cycle.
\end{prop}
\begin{proof} By definition, if $X,\,B$ are smooth projective and $\phi: X\rightarrow B$ is a dominant  morphism of relative dimension $d$ and  index $1$, there exists a cycle $\Gamma\in {\rm CH}^d(X)$ such that $\phi_{*}\Gamma=B$ in ${\rm CH}^0(B)$. The cycle  $\Gamma$  has a cohomology class $[\Gamma]$ acting by cup-product on integral cohomology of $X$ and we have by the projection formula
\begin{eqnarray}\label{eqactiongama} \phi_* ([\Gamma]\cup \phi^*\alpha)=\alpha\end{eqnarray}
for any $\alpha\in H^*(B,\mathbb{Z})$.

We now consider the case where $B$ is a curve $C$ and $\phi=\phi_X$ induces an isomorphism
$${\rm Alb}(X)\cong J(C).$$
 If  the index of $\phi_X$ is $1$, the construction above gives a cycle
$$\Gamma'= (\phi_X,Id_X)_*\Gamma\in {\rm CH}_1(C\times X)$$
satisfying  by (\ref{eqactiongama})
\begin{eqnarray}\label{eqgammaprime} \phi_{X*}\circ [\Gamma']_*={\rm Id}: H^1(C,\mathbb{Z})\rightarrow H^1(C,\mathbb{Z}).
\end{eqnarray}
As before we use a Poincar\'{e} divisor $Z_C\in {\rm CH}^1(J(C)\times C)$ such that
$$Z_{C*}: H_1(J(C),\mathbb{Z})\rightarrow H_1(C,\mathbb{Z})$$
is the natural isomorphism, inverse of ${\rm alb}_{C*}:  H_1(C,\mathbb{Z})\rightarrow H_1(J(C),\mathbb{Z})$.
It follows from (\ref{eqgammaprime}) that
the cycle
$$\Gamma'':=Z_{C}\circ \Gamma' \in{\rm CH}^n(J(C)\times X),\, n={\rm dim}\,X$$
has the property that
$$[\Gamma'']_*: H_1(J(C),\mathbb{Z})\rightarrow H_1(X,\mathbb{Z})$$
satisfies $\phi_{X*}\circ [\Gamma'']_*=Z_{C*}$. As $\phi_{X*}$ induces by assumption an isomorphism
$$H_1(X,\mathbb{Z})_{\rm tf}\cong H_1(C,\mathbb{Z})$$
we conclude that $[\Gamma'']_*: H_1(J(C),\mathbb{Z})=H_1({\rm Alb}(X),\mathbb{Z})\rightarrow H_1(X,\mathbb{Z})_{\rm tf}$ is the inverse of the isomorphism ${\rm alb}_{X*}: H_1(X,\mathbb{Z})_{\rm tf}\cong H_1({\rm Alb}(X),\mathbb{Z})$. Hence $\Gamma''$ is a universal $0$-cycle for $X$.
\end{proof}
A particular case of Proposition \ref{coroaprouver} is the case of rationally connected fibrations over a curve, which is considered in \cite{colliotTh} and \cite{voisinArbarello}. They admit universal $0$-cycles by \cite{ghs} and Proposition \ref{coroaprouver}.

We now show that the implication of Proposition \ref{coroaprouver} is strict by constructing  an example of a smooth projective surface $S$ with a morphism $\phi:S\rightarrow C$ to a curve, inducing an isomorphism $\phi_*:{\rm Alb}(S)\cong J(C)$, such that the index of $\phi$ is $d>1$ and $S$ admits a universal $0$-cycle. We start with any curve $C$ of genus $>0$  and choose a smooth curve \begin{eqnarray}
\label{eqsmoothcurve} \Gamma \subset C\times C
\end{eqnarray}
 such that the second projection $p:\Gamma\rightarrow C$ has degree $d$ and $\Gamma^*:H^1(C,\mathbb{Z})\rightarrow H^1(C,\mathbb{Z})$ is the identity. This exists for $d$ large enough, by taking (for example) $\Gamma$ to be a smooth member of the linear system
$|\mathcal{O}_{C\times C}((d-1)(c\times C)+d' (C\times c')+\Delta_C)|$ for some points $c,\,c'\in C$ and integers $d,\,d'$ large enough, where $\Delta_C$ is the diagonal of $C$.
We now choose an embedding $i$ of  $C$ in $\mathbb{P}^3$, which gives an embedding $({\rm Id}_C,i)$ of $C\times C$, hence of $\Gamma$, in $C\times \mathbb{P}^3$. For some very ample line bundle $L$ on $C$, we choose the surface $S\subset C\times \mathbb{P}^3$ to be a very general complete intersection of members of the linear system $|L\boxtimes \mathcal{O}_{\mathbb{P}^3}(d)|$ containing $({\rm Id}_C,i)(\Gamma)$.  The first projection $\phi:S\rightarrow C$ induces by the Lefschetz theorem on hyperplane sections an isomorphism
${\rm Alb}(S)\cong J(C)$. We have by construction $C\times \Gamma\subset C\times S$ and $C\times \Gamma$ contains the transpose  ${^t\Gamma}_{p'}\cong \Gamma$ of the graph of the morphism $p':\Gamma\rightarrow C$ given by the first projection in (\ref{eqsmoothcurve}). Hence  ${^t\Gamma}_{p'}$ is contained in $C\times S$ and using the fact that
$$\phi_{\mid \Gamma}=p,\,{\rm pr}_{1\mid {^t\Gamma}_{p'}}=p',$$  its image in $C\times C$ under $({\rm Id}_C,\phi)$ is the original curve $\Gamma\subset C\times C$. Hence it acts as the identity on $H^1(C,\mathbb{Z})$, so $S$ has a universal $0$-cycle.
Finally we argue as in \cite{lopez} for example, to conclude that for very general $S$, the N\'{e}ron-Severi group of $S$ is generated by a fiber of $\phi$, the line bundle $\mathcal{O}_{\mathbb{P}^3}(1)$ and the class of the curve $({\rm Id}_C,i)(\Gamma)$.
This immediately implies that the index of $\phi$ is $d$, since $\Gamma$ has degree $d$ over $C$.
\section{A $3$-dimensional  example for Question \ref{questionCT}\label{secexCT}}
We give in this section an example of a smooth projective $3$-fold $X$ with representable ${\rm CH}_0$ and no universal $0$-cycle, thus proving Theorem \ref{theo}.
Let $S$ be a projective $K3$ surface over the complex numbers  equipped with a fixed point free antisymplectic involution $g$. We denote by $\Sigma:=S/g$ the corresponding Enriques surface.
 Let $E$ be an elliptic curve with a translation $t_\xi$ by a point $\xi$ of order $2$. We denote $E_\xi:=E/t_\xi$.

 The variety we will consider is

 \begin{eqnarray}\label{eqX} X:= (E\times S)/(t_\xi,g).\end{eqnarray}
 It admits two natural morphisms
 \begin{eqnarray}\label{eqtowmorph} p_{E_\xi}: X\rightarrow E_\xi,\,p_\Sigma: X\rightarrow \Sigma.\end{eqnarray}

 \begin{lemm}\label{lealbetchow} (i) The morphism $p_{E_\xi}$ induces an isomorphism
 $$ p_{E_\xi*}:{\rm CH}_0(X)\rightarrow  {\rm CH}_0(E_\xi).$$ In particular ${\rm CH}_0(X)$ is representable.

 (ii) One has ${\rm Alb}(X)\cong E_\xi$ and the morphism $p_{E_\xi}$ identifies (up to translation)  with the Albanese map of $X$.
 \end{lemm}

\begin{proof} The morphism $p_{E_\xi}: X\rightarrow E_\xi=E/t_\xi$ has all its fibers isomorphic to   the $K3$ surface $S$, which is connected and simply connected. It follows that $p_{E*}:H_1(X,\mathbb{Z})\rightarrow  H_1(E_\xi,\mathbb{Z})$ is an isomorphism, which proves (ii). As (ii) is proved, in order to prove  (i), it suffices to prove that   the kernel
${\rm CH}_0(X)_{\rm alb}$
 of the Albanese morphism
 ${\rm CH}_0(X)_{\rm hom}\rightarrow {\rm Alb}(X)$ is trivial.
 As this kernel has no torsion by \cite{roitmantorsion}, it injects by pull-back in the invariant part under $(t_\xi,g)$ of the group  ${\rm CH}_0(E\times S)_{\rm alb}$.
 It thus  suffices to show that the involution $(t_\xi,g)$ acts by $-{\rm Id}$ on the group  ${\rm CH}_0(E\times S)_{\rm alb}\otimes \mathbb{Q}$. The involution $t_\xi$ acts as the identity on ${\rm CH}_0(E)\otimes \mathbb{Q}$. The involution $g$ acts by $-{\rm Id}$ on the group ${\rm CH}_0(S)_{\rm hom}$ by \cite{blochkaslieberman}.
  With rational coefficients, choosing a $0$-cycle $o_S$ of degree $1$ which is $g$-invariant, we have a decomposition
  \begin{eqnarray}\label{eqdecompdu1509} {\rm CH}_0(E\times S)\otimes \mathbb{Q}={\rm CH}_0(E)_\mathbb{Q}\otimes o_S \oplus \\
  \nonumber
  {\rm Im}\,({\rm CH}_0(E)_\mathbb{Q}\otimes {\rm CH}_0(S)_{\rm hom}\rightarrow  {\rm CH}_0(E\times S)_\mathbb{Q}) ,\end{eqnarray}
  and the group ${\rm CH}_0(E\times S)_{\rm alb}\otimes \mathbb{Q}$ is contained in the second summand of the decomposition (\ref{eqdecompdu1509}), on which $(t_\xi,g)$ acts by $-{\rm Id}$,
  which concludes the proof.
\end{proof}

We now prove the main result.
\begin{theo} \label{theoexample} For a fixed $K3$ surface $S$ with involution $g$ over an Enriques surface $\Sigma$, there exists an elliptic curve  $E$ with a translation $t_\xi$ of order $2$, such that  the variety $X$ constructed above does not admit a universal $0$-cycle.
\end{theo}
\begin{proof} Given $S,\,E$ and $\xi$, let us assume that $X$ has a universal $0$-cycle.  By Lemma \ref{lealbetchow}, a universal $0$-cycle for $X$ is given by a $1$-cycle
$$\Gamma\in {\rm CH}_1(E_\xi\times X)
$$
which has the property that $[\Gamma]_*: H_1(E_\xi,\mathbb{Z})\rightarrow H_1(X,\mathbb{Z})_{\rm tf}$ is inverse to the isomorphism $p_{E_\xi*}:H_1(X,\mathbb{Z})_{\rm tf}\rightarrow  H_1(E_\xi,\mathbb{Z})$, that is
\begin{eqnarray}\label{eqcominveEeta} p_{E_\xi*}\circ [\Gamma]_*={\rm Id}: H_1(E_\xi,\mathbb{Z})\rightarrow H_1(E_\xi,\mathbb{Z}).
\end{eqnarray}
Dualizing (\ref{eqcominveEeta}) and passing to $\mathbb{Z}/2$-coefficients, we get that

\begin{eqnarray}\label{eqcominveEetadualmod2}  [\Gamma]^*\circ p_{E_\xi}^*={\rm Id}: H^1(E_\xi,\mathbb{Z}/2)\rightarrow H^1(E_\xi,\mathbb{Z}/2).
\end{eqnarray}

The \'{e}tale double covers $S\rightarrow \Sigma$, $E\rightarrow E_\xi$, $q=(p_{E_\xi},p_\Sigma):X\rightarrow E_\xi\times \Sigma$ are given by elements
$$\sigma_\Sigma\in  H^1(\Sigma,\mathbb{Z}/2),\,\sigma_{E_\xi}\in  H^1(E_\xi,\mathbb{Z}/2),\,\sigma_X\in  H^1(E_\xi\times \Sigma,\mathbb{Z}/2)$$ that are characterized by the fact that they are nonzero and they are annihilated respectively  by the pull-backs maps
$$H^1(\Sigma,\mathbb{Z}/2)\rightarrow H^1(S,\mathbb{Z}/2),\,H^1(E_\xi,\mathbb{Z}/2)\rightarrow H^1(E,\mathbb{Z}/2),\, q^*:H^1(E_\xi\times \Sigma,\mathbb{Z}/2)\rightarrow H^1(X,\mathbb{Z}/2).$$
Recalling the notation  of  (\ref{eqtowmorph}), we claim that
\begin{eqnarray}\label{eqpeutetre} \sigma_X={\rm pr}_{E_\xi}^*(\sigma_{E_\xi})+{\rm pr}_\Sigma^*(\sigma_\Sigma)\,\,{\rm in}\,\,H^1(E_\xi\times \Sigma,\mathbb{Z}/2),
\end{eqnarray}
where ${\rm pr}_{E_\xi},\,{\rm pr}_{\Sigma}$ are the projections from $E_\xi\times \Sigma$ to its factors, so that
\begin{eqnarray}\label{eqpetqeta} p_{E_\xi}={\rm pr}_{E_\xi}\circ q,\,\,p_{\Sigma}={\rm pr}_\Sigma\circ q.
\end{eqnarray}
In order to prove (\ref{eqpeutetre}), we write
$$\sigma_X={\rm pr}_{E_\xi}^*(\alpha)+{\rm pr}_\Sigma^*(\beta)\,\,{\rm in}\,\,H^1(E_\xi\times \Sigma,\mathbb{Z}/2)$$
for some classes $\alpha\in H^1(E_\xi,\mathbb{Z}/2),\,\beta\in H^1(\Sigma,\mathbb{Z}/2)$. It follows that
for any $e\in E_\xi,\,s\in \Sigma$,
$$\sigma_{X\mid \{e\}\times \Sigma}=\beta,\,\sigma_{X\mid E_\xi\times \{s\}}=\alpha.$$
Finally we observe that the fiber of $p_{E_\xi}$ over $e$ is the double cover of $\Sigma$ isomorphic to $S$ and the fiber of $p_\Sigma$ over $s$ is the double cover of $E_\xi$ isomorphic to $E$, so that
$$\beta=\sigma_\Sigma,\,\alpha=\sigma_{E_\xi},$$
which proves the claim.

The meaning of (\ref{eqpeutetre}) is that
\begin{eqnarray}\label{eqpeutetremean}q^*({\rm pr}_{E_\xi}^*(\sigma_{E_\xi})+{\rm pr}_\Sigma^*(\sigma_\Sigma))=0\,\,{\rm in}\,\,H^1(X,\mathbb{Z}/2),\end{eqnarray}
which  thanks to (\ref{eqpetqeta}) rewrites as
\begin{eqnarray}\label{eqrewrite} p_{E_\xi}^*(\sigma_{E_\xi})+{p}_\Sigma^*(\sigma_\Sigma)=0\,\,{\rm in}\,\,H^1(X,\mathbb{Z}/2).
\end{eqnarray}
 Let $$\Gamma_E:=({\rm Id}_{E_\xi}, p_{E_\xi})_*(\Gamma)\in{\rm CH}_1(E_\xi\times E_\xi),\,\Gamma_\Sigma:=({\rm Id}_{E_\xi}, p_\Sigma)_*(\Gamma)\in {\rm CH}_1(E_\xi\times \Sigma).$$
 It follows from (\ref{eqrewrite}) that
 $$\Gamma^*(p_{E_\xi}^*(\sigma_{E_\xi}))+\Gamma^*({p}_\Sigma^*(\sigma_\Sigma))=0\,\,{\rm in}\,\,H^1(E_\xi,\mathbb{Z}/2),$$
 which can be written as
 \begin{eqnarray}\label{eqfatal}  \Gamma_E^*(\sigma_{E_\xi})+\Gamma_\Sigma^*(\sigma_\Sigma)=0\,\,{\rm in}\,\,H^1(E_\xi,\mathbb{Z}/2).
 \end{eqnarray}
 By equation (\ref{eqcominveEetadualmod2}),  $\Gamma_E^*$ acts as the identity on $H^1(E_\xi,\mathbb{Z}/2)$, hence
 we finally get from (\ref{eqfatal})
 \begin{eqnarray} \label{eqfatal1} \sigma_{E_\xi}=\Gamma_\Sigma^*(\sigma_\Sigma)\,\,{\rm in}\,\,H^1(E_\xi,\mathbb{Z}/2).
 \end{eqnarray}

 The  results of \cite[Propositions 1.1 and  2.1]{benoistOttem}  tell now that  there exist $E_\xi,\,\sigma_{E_\xi}$ such that (\ref{eqfatal1})  holds for no  $1$-cycle $\Gamma_\Sigma\in{\rm CH}_1(E_\xi\times \Sigma)$, (see also Proposition \ref{propBOwithnewproof} for  an alternative proof and generalization of this statement).
 The analysis above thus shows that for this choice of $S$ and of pair $(E_\xi,\,\sigma_{E_\xi})$ (determining a double cover $E$ with translation $t_\xi$), $X$ does not have a universal $0$-cycle.
\end{proof}
\begin{coro} \label{coroindex} For $S$ fixed and  $E$ very general, the index of $p_{E_\xi}: X\rightarrow E_\xi$ is $2$.
\end{coro}
\begin{proof} The index of $p_{E_\xi}$ is either $1$ or $2$, since the image of a curve $E\times s$ in $X$ for some $s\in S$ has degree $2$ over $E_\eta$. However, the index cannot be $1$, as  otherwise Proposition  \ref{coroaprouver} would provide a contradiction with  Theorem \ref{theoexample}.
\end{proof}
We finish with the following statement, which combined with Corollary \ref{coroindex} shows that the index of $p_{E_\xi}$ is not dictated by topological or Hodge-theoretic reasons.
\begin{lemm}\label{lehodgeindex} There exists an integral Hodge class $\alpha\in H^4(X,\mathbb{Z})$ of degree $4$ on $X$ such that
$p_{E_\xi*}\alpha=1_{E_\xi}$ in $H^0(E_\xi,\mathbb{Z})$.
\end{lemm}
\begin{proof} We note that, because $H^2(X,\mathcal{O}_X)=0$, any integral Betti cohomology class $\alpha\in H^4(X,\mathbb{Z})$ is Hodge. So it suffices to prove the existence of an integral Betti cohomology class $\alpha\in H^4(X,\mathbb{Z})$ such that $p_{E_\xi*}\alpha=1_{E_\xi}$ in $H^0(E_\xi,\mathbb{Z})$. This statement is topological, so it suffices to prove the result for a specific elliptic curve $E$ with involution $t_\xi$. We choose a $K3$ surface with involution $g$ over an Enriques surface satisfying the property that ${\rm NS}(S)$ is $g$-invariant. We now observe that $S$  has elliptic pencils, which must be   globally invariant under $g$. The action of $g$ on the base $\mathbb{P}^1$ of the pencil has two fixed points, and each one produces a possibly singular elliptic curve $E\subset S$, invariant under $g$. As $g$ acts without fixed  point on $E$, $E$ must be smooth and $g$ acts as an order $2$  translation  $t_\xi$ on it. Indeed, the only possibility for a singular fixed fiber $E$ would be that $E$ is the union of two smooth rational curves meeting in two points, where $g$ exchanges the two components. This however would contradict the fact that ${\rm NS}(S)$ is $g$-invariant. We now consider the diagonal inclusion of $E$ in $E\times S$. It is invariant under $(t_\xi, g)$ and we get a curve $E/ (t_\xi, g)\subset (E\times S)/(t_\xi,g)=X$, which is isomorphic to $E_\xi:=E/t_\xi$ via $p_{E_\xi}$. The class $\alpha:=[E/ (t_\xi, g)]\in H^4(X,\mathbb{Z})$ is the desired one.
\end{proof}
\subsection{On a  result of Benoist and Ottem \label{secBO}}
Our proof of Theorem \ref{theo} relied on the results of Benoist and Ottem in \cite{benoistOttem}. We give in this section a short proof and a mild generalization of  the statement in \cite{benoistOttem} that we have been using.
\begin{prop} \label{propBOwithnewproof} Let $W$ be a smooth projective complex variety of dimension $d$,  and let $0\not=\beta\in H^1(W,\mathbb{Z}/k)$ for some prime  integer $k>1$. Then there exist only countably many elliptic curves $E$, such that there exists a correspondence $\Gamma\in {\rm CH}^d(E\times \Sigma)$ with
$$[\Gamma]^*\beta\not=0\,\,{\rm in }\,\,H^1(E,\mathbb{Z}/k).$$
\end{prop}
\begin{proof} We argue by contradiction and, using standard spreading arguments, we get that there exist a smooth projective curve $B$, a smooth nonisotrivial elliptically fibered   surface $f:T\rightarrow B$ and a correspondence
$\Gamma\in {\rm CH}^d(T\times W)$, such that the class
$\beta':=[\Gamma]^*\beta\in H^1(T,\mathbb{Z}/k)$ does not vanish on the smooth fibers $T_b$ of $f$. We can furthermore assume that $f$ has a section, by performing an extra  base-change $B'\rightarrow B$ if necessary.
The class $\beta'$ induces an étale cyclic cover
$e: T'\rightarrow T$ of degree $k$,
and the map $f':=f\circ e:T'\rightarrow B$ has connected elliptic general fiber $T'_b$, obtained as the  étale degree $k$ cyclic cover of the fiber $T_b$ induced by the class $\beta'_{\mid T_b}$ (which is  nontrivial by assumption). The two surfaces satisfy
\begin{eqnarray} \label{eqhodgeequal} H^i(T,\mathcal{O}_T)\cong H^i(T',\mathcal{O}_{T'}),\,i=0,\,1,\,2.\end{eqnarray}
This follows indeed from the fact that the order $k$ automorphism  $i$ of $T'$ over $T$ acts by a finite order translation of the smooth fibers $T'_b$. It thus  acts as the identity on ${\rm CH}_0(T'_b)_\mathbb{Q}$ for a general $b$, hence also  on ${\rm CH}_0(T')_\mathbb{Q}$, which implies  by \cite{mumford} that it acts as the identity on $H^i(T',\mathcal{O}_{T'})$.
Formula (\ref{eqhodgeequal}) gives that $$\chi(T',\mathcal{O}_{T'})=\chi(T,\mathcal{O}_{T}),$$ while we also have

$$\chi(T',\mathcal{O}_{T'})=k\chi(T,\mathcal{O}_{T})$$
since $e:T'\rightarrow T$ is étale of degree $k$.
We thus conclude that $\chi(T,\mathcal{O}_{T})=0$. However the Kodaira canonical bundle formula \cite{canobun1}, \cite{canobun2} shows that a surface $T$ admitting a nonisotrivial elliptic fibration with a section cannot have  $\chi(T,\mathcal{O}_{T})=0$.
\end{proof}

Proposition \ref{propBOwithnewproof} applies for example to quintic  Godeaux surfaces, quotients of quintic surfaces $S$ in $\mathbb{P}^3$ by the  action of a certain  order $5$ automorphism $g$ acting freely. It is proved in  \cite{voisingodeaux} that such a surface $\Sigma$ has trivial ${\rm CH}_0$ group, hence the whole argument described previously would work as well, thanks to Proposition \ref{propBOwithnewproof}, for a general quotient $X=(E\times S)/(t_\xi,g)$, where $E$ is very general and  $t_\xi$ is a translation by a point of order $5$. Such an $X$ has representable ${\rm CH}_0$ group and has no universal $0$-cycle.

\section{The case of surfaces\label{secsurfcaes} }
A ruled surface, and more generally a rationally connected fibration over a curve,  admits a universal $0$-cycle, thanks to Lemma \ref{coroaprouver} (see also  \cite{colliotTh}, \cite{voisinArbarello}). We also have the following.
\begin{prop} \label{propsourface} Let $S$ be a smooth projective surface with $p_g(S)=0$ and $q(S)\not=0$.  Assume the morphism
$$\phi_S:S\rightarrow C,\,J(C)={\rm Alb}(S)$$ given by Lemma \ref{lestructure} has the property that the class
$[S_c]\in H^2(S,\mathbb{Z})_{\rm tf}$ of the fiber $S_c=\phi_S^{-1}(c)$ is primitive. Then the index of $\phi_S$ is $1$ and  $S$ admits a universal $0$-cycle.
\end{prop}
\begin{proof} By Proposition \ref{coroaprouver}, the first statement implies the second one.  Let us now prove that under our assumptions, the index of $\phi_S$ is $1$.  As the class $[S_c]$ is primitive, Poincar\'{e} duality on $H^2(S,\mathbb{Z})$ tells that there exists a class $\gamma\in H^2(S,\mathbb{Z})$ such that $\langle \gamma,[S_c]\rangle=1$. As $p_g(S)=0$, the  theorem on $(1,1)$-classes tells that $H^2(S,\mathbb{Z})\cong {\rm NS}(S)$, hence $\gamma=[D]$ for some divisor class $D \in {\rm NS}(S)$. Thus $D$ has degree $1$ along the fibers of $\phi_S$.
\end{proof}
We now consider the non-uniruled case and  first recall the following result due to Beauville \cite{beauville}.
\begin{theo} \label{theobeau} \cite[Proposition 11]{beauville} Let $S$ be a non-uniruled smooth projective surface such that $p_g(S)=0,\,q(S)\not=0$. Then $S$ is birational to a quotient $(C_1\times C_2)/G$, where $G$ is a finite group acting by automorphisms of both $C_1$ and $C_2$, without fixed points on $C_1\times C_2$. Furthermore, at least one of the two curves is elliptic, and we have (up to exchanging $C_1$ and $C_2$)
$$ C_1/G\cong E,\,\,C_2/G\cong \mathbb{P}^1,$$
where $E$ is elliptic and ${\rm Alb}(S)=J(E)\cong E$.
\end{theo}
\begin{rema}{\rm We can assume that $G$ acts faithfully on both $C_1$ and $C_2$, because otherwise, denoting by
$H$ the kernel of morphism $G\rightarrow {\rm Aut}(C_1)$, we can replace $C_2$ by $C_2/H$ and $\overline{G}:=G/H$ acts on $C_1$ and $C'_2$ with $(C_1\times C_2)/G\cong (C_1\times C_2/H)/\overline{G}$. Assuming the faithfulness condition, if $C_2$ is not elliptic, then the group $G$ is commutative, because $C_1$ is elliptic and $C_1/G$ is also elliptic.}
\end{rema}
We now have
\begin{prop}\label{prosurfcascyclique} Assume $S$ is  as above and the group $G$ is cyclic. Then the index of $\phi_S$ is equal to $1$, hence $S$ has a universal $0$-cycle.
\end{prop}
\begin{proof} Again the second statement follows from the first. Let $G$ be cyclic generated by $g$, and assume the order of $g$ is $d$. Using the notation of Theorem \ref{theobeau}, we know that
$C_2/G\cong\mathbb{P}^1$. An irreducible degree $d$ cyclic cover of $\mathbb{P}^1$  has an affine version with  equation
$$ u^d=(t-t_1)^{a_1}\ldots (t-t_k)^{a_k},$$
where the gcd of $\{d,\,a_i,\,i=1,\ldots,\,k\}$ is $1$ since otherwise the curve is not irreducible. Let $r_i:={\rm gcd}(d,a_i)$. After normalization, the ramification order of $C_2\rightarrow \mathbb{P}^1$ over $t_i$ is $d/r_i$ and $d/r_i$ is the order of
the subgroup $G_i\subset G$ fixing any point $c_i\in C_2$ over $t_i\in \mathbb{P}^1$. The surface $S$ contains the quotient $(C_1\times \{c_i\})/G_i$, which has degree $r_i$ over $C_1/G$. As the gcd of the  set $\{r_i\}$  is $1$, we conclude that the index of $\phi_S$ is $1$.
\end{proof}

\end{document}